\documentclass[12pt,reqno]{amsart}

\usepackage[a4paper,margin=1in]{geometry}
\usepackage{amsmath,amssymb,amsthm,mathtools}
\usepackage{bbm}
\usepackage{enumitem}
\usepackage{hyperref}
\usepackage[normalem]{ulem}

\usepackage{mathrsfs}

\usepackage{pdflscape}
\usepackage{subcaption}
\usepackage{caption}
\usepackage{tikz, float, tikzscale}
\usepackage{pgfplots}
\pgfplotsset{compat=1.18}

\usepackage[T1]{fontenc}
\usepackage{libertinus,libertinust1math}

\theoremstyle{definition}
\newtheorem{defi}{Definition}

\theoremstyle{plain}
\newtheorem{thm}{Theorem}
\newtheorem*{thm*}{Theorem}
\newtheorem{prop}{Proposition}
\newtheorem{lemma}{Lemma}

\newtheorem{coro}{Corollary}

\DeclareMathOperator{\diam}{diam}

\title{Prescribed realisation of longest runs in continued fractions}
\author{Ying Wai Lee}

\begin{document}

\begin{abstract}
    Exceptional longest-run behaviour in continued fraction expansions is studied through the interaction between fixed-symbol runs and the overall longest run. For every prescribed partial quotient value and every admissible growth scale, a full Hausdorff dimensional set of irrational numbers is constructed on which the longest run of the prescribed value has exactly the prescribed asymptotic growth and, for every initial length, uniquely realises the overall maximum. It follows that the symbol responsible for the overall longest run can be fixed in advance without any loss of Hausdorff dimension. Thus, the known full-dimensional exceptional-set results for fixed-symbol longest-run growth and for overall longest-run growth are simultaneously strengthened, while the maximising symbol in the overall problem is shown to be fully prescribable.
\end{abstract}
\maketitle

\section{Introduction}

Continued fractions of irrational numbers in the unit interval are a classical topic in number theory and dynamical systems. For any $x\in[0,1)\setminus\mathbb{Q}$, there exists a unique sequence of positive integers $(a_n)_{n\in\mathbb{N}}$ such that $x$ admits the continued fraction expansion:
\begin{align*}
    x=[a_1,a_2,a_3,\dots]
    \coloneqq\frac{1}{\displaystyle a_1+\frac{1}{\displaystyle a_2+\frac{1}{a_3+\cdots}}},
\end{align*}
where for any $n\in\mathbb{N}$, $a_n\coloneqq a_n(x)\in\mathbb{N}$ is referred to as the $n$-th partial quotient of $x$. Define the Gauss map $T:[0,1]\to[0,1]$ by $T(0)\coloneqq1$ and for any $x\in(0,1]$,
\begin{align*}
    T(x)\coloneqq\frac{1}{x}-\left\lfloor\frac{1}{x}\right\rfloor.
\end{align*}
The partial quotients of an irrational number can be generated by iterating the Gauss map. For any $n\in\mathbb{N}$ and $x\in[0,1)\setminus\mathbb{Q}$, the $n$-th partial quotient of $x$ is given by
\begin{align*}
    a_n(x)=\left\lfloor \frac{1}{T^{n-1}(x)}\right\rfloor.
\end{align*}

Longest-run problems for continued fraction expansions have been studied from fixed-symbol and overall points of view. Define, for any $n\in\mathbb{N}$, the fixed-symbol longest-run function $L_n:[0,1)\setminus\mathbb{Q}\times\mathbb{N}\to\mathbb{N}\cup\{0\}$ by, for any $x\in [0,1)\setminus\mathbb{Q}$ and $\lambda\in\mathbb{N}$,
\begin{align*}
    L_n(x,\lambda)
    \coloneqq\max_{0\leq j\leq n-1}\sum_{k=1}^{n-j}\prod_{i=1}^k\mathbf{1}_{a_{j+i}(x)=\lambda};
\end{align*}
equivalently, $L_n(x,\lambda)$ is the maximum length of a consecutive run of the value $\lambda$ within the first $n$ partial quotients of $x$. Define, for any $n\in\mathbb{N}$, the overall longest-run function $R_n:[0,1)\setminus\mathbb{Q}\to\mathbb{N}$ by, for any $x\in[0,1)\setminus\mathbb{Q}$,
\begin{align*}
    R_n(x)
    \coloneqq \max_{\lambda\in\mathbb{N}} L_n(x,\lambda);
\end{align*}
equivalently, $R_n(x)$ is the maximum length of a consecutive run over all values within the first $n$ partial quotients of $x$. For example, one obtains $L_{30}(\pi-3,1)=3$ and $L_{30}(\pi-3,2)=R_{30}(\pi-3)=4$ from the first 30 partial quotients of $\pi-3$:
\begin{align*}
    \pi-3=[7,15,1,292,
    \underline{1,1,1},
    2,1,3,1,14,2,1,1,
    \underline{2,2,2,2},
    1,84,2,1,1,15,3,13,1,4,2,\ldots]
\end{align*}
where the underlined blocks show the longest runs.

Longest-run problems originate in classical probability theory. For independent symbolic processes, the logarithmic order of the longest run goes back to the Erd\H{o}s--R\'enyi law of large numbers~\cite{ErdosRenyi1970}; sharper refinements for the longest head-run were later obtained by Erd\H{o}s and R\'ev\'esz~\cite{ErdosRevesz1975}. Since then, longest-run problems have become a standard object in the study of rare events and extreme symbolic patterns. In the continued-fraction system equipped with the natural absolutely continuous reference measure, such as the Gauss measure, the partial quotients form a stationary symbolic process which is mixing but not independent, and whose one-symbol distribution is highly non-uniform. These dynamical and probabilistic features make longest-run problems for continued fraction expansions a natural analogue of the classical run-length problem, now in the setting of a non-uniformly expanding dynamical system with a countable alphabet.

Typical longest-run behaviours of partial quotients, in the sense of full Lebesgue measure, have been studied in the recent literature. Song--Zhou~\cite[Theorem~1.1]{SongZhou2020} established the first-order behaviour for the fixed-symbol longest-run function: for any $\lambda\in\mathbb{N}$ and almost every $x\in[0,1)\setminus\mathbb{Q}$,
\begin{align*}
    \lim_{n\to\infty}\frac{L_n(x,\lambda)}{\log{n}/\log{\tau_\lambda}}
    =\frac12,
    \qquad
    \tau_\lambda\coloneqq \frac{\lambda+\sqrt{\lambda^2+4}}{2}.
\end{align*}
Wang--Wu~\cite[Theorem~1.3]{WangWu2011} established the first-order behaviour of the overall longest-run function: for almost every $x\in[0,1)\setminus\mathbb{Q}$,
\begin{align*}
    \lim_{n\to\infty}\frac{R_n(x)}{\log{n}/\log{\varphi}}
    =\frac12,
    \qquad
    \varphi\coloneqq\frac{1+\sqrt5}{2}.
\end{align*}
Recently, Lee~\cite[Theorem~2]{Lee2026} obtained eventual two-sided additive refinements of the above first-order laws with double-logarithmic error terms: for any $\lambda\in\mathbb{N}$, $c>1/2$, almost every $x\in[0,1)\setminus\mathbb{Q}$ and sufficiently large $n\in\mathbb{N}$,
\begin{align*}
    \left|L_n(x,\lambda)-\frac{\log{n}}{2\log{\tau_\lambda}}\right|
    \leq \frac{c\log{\log{n}}}{\log{\tau_\lambda}},
    \qquad
    \left|R_n(x)-\frac{\log{n}}{2\log{\varphi}}\right|
    \leq \frac{c\log{\log{n}}}{\log{\varphi}}.
\end{align*}

A central topic in the study of atypical behaviour of longest-run functions concerns the set of numbers whose longest-run growth deviates from the first-order law satisfied by almost every number. Recent work has focused on the Hausdorff dimension of sets defined by prescribed growth rates of longest-run functions. Song--Zhou~\cite[Corollary~1.3]{SongZhou2020} proved that for any $\lambda\in\mathbb{N}$ and $\alpha\in[0,+\infty]$,
\begin{align*}
    \dim{
    \left\{x\in[0,1)\setminus\mathbb{Q}:
    \lim_{n\to\infty}\frac{L_n(x,\lambda)}{\log{n}/\log{\tau_\lambda}}=\alpha\right\}}=1.
\end{align*}
Wang--Wu~\cite[Theorem~1.3]{WangWu2011} obtained a more general exact-growth result for the overall longest-run: for any unbounded non-decreasing sequence of positive integers $(\delta_n)_{n\in\mathbb{N}}$, if
\begin{align*}
    \lim_{n\to\infty}\frac{\delta_{n+\delta_n}}{\delta_n}=1
\end{align*}
then 
\begin{align*}
    \dim{
    \left\{x\in[0,1)\setminus\mathbb{Q}:
    \lim_{n\to\infty}\frac{R_n(x)}{\delta_n}=1\right\}}=1.
\end{align*}
Related $\limsup$- and $\liminf$-type exceptional sets were also studied in the literature~\cite{SongZhou2020,TanZhou2025,WangWu2011}.

The results above prescribe either the longest run of a fixed symbol, or the longest run after maximising over all symbols. However, they do not address the interaction between these two prescriptions. In particular, if the growth of fixed-symbol longest-run is prescribed for a fixed value, it is not automatic that this prescribed symbol also realises the overall longest run. Conversely, prescribing the growth of overall longest-run does not specify which partial quotient value is responsible for the maximal blocks.

In the present work, a full-dimensional exceptional-set theorem is established, in which the growth of a fixed-symbol longest run is prescribed and the overall longest run is uniquely realised by that same prescribed symbol at all times. As consequences, the single theorem is a genuine simultaneous refinement of the results of Song–Zhou and Wang–Wu: not only can the fixed-symbol and overall longest-run growth be prescribed on a full-dimensional set, but the overall longest run can be forced to be uniquely realised at any preassigned partial quotient value.

\section{Main results}

The following definition isolates, in functional notation, the exact hypothesis appearing in the above result of Wang--Wu~\cite[Theorem~1.3]{WangWu2011}.
\begin{defi}
Let $f:\mathbb{N}\to\mathbb{N}$ be a function. $f$ is said to be admissible, if $f$ is non-decreasing, unbounded, and
\begin{align*}
    \lim_{n\to\infty}\frac{f(n+f(n))}{f(n)}=1.
\end{align*}
\end{defi}

Theorem~\ref{thm:main-sn-final} is the main theorem of the present work.
\begin{thm}
\label{thm:main-sn-final}
For any $\lambda\in\mathbb{N}$ and admissible $f:\mathbb{N}\to\mathbb{N}$, 
\begin{align*}
    \dim{\mathfrak{E}_{\lambda,f}}=1,
\end{align*}
where
\begin{align*}
    \mathfrak{E}_{\lambda,f}
    &\coloneqq
    \left\{x\in[0,1)\setminus\mathbb{Q}:
    \lim_{n\to\infty}\frac{L_n(x,\lambda)}{f(n)}=1
    \text{ and }\right.\\
    &\qquad\qquad\left.\text{for any $n\in\mathbb{N}$ and $\mu\in\mathbb{N}\setminus\{\lambda\}$, }\frac{L_n(x,\lambda)}{L_n(x,\mu)}>1\right\}.
\end{align*}
\end{thm}

The first corollary records the fixed-symbol exact-growth consequence of Theorem~\ref{thm:main-sn-final}. Compared with the fixed-symbol exceptional-set theorem of Song--Zhou, it gives a refinement in which the prescribed growth scale is allowed to be any admissible function. Thus this corollary isolates the fixed-symbol part of the main theorem after the additional dominance condition is discarded.
\begin{coro}
\label{label:coro1}
For any $\lambda\in\mathbb{N}$ and admissible $f:\mathbb{N}\to\mathbb{N}$, 
\begin{align*}
    \dim{\left\{x\in[0,1)\setminus\mathbb{Q}:
    \lim_{n\to\infty}\frac{L_n(x,\lambda)}{f(n)}=1\right\}}=1.
\end{align*}
\end{coro}

The second corollary records the overall exact-growth consequence of Theorem~\ref{thm:main-sn-final}. Compared with the theorem of Wang--Wu for the overall longest-run function, it recovers the same full-dimensional exact-growth statement. Thus this corollary shows that the main theorem contains the known overall result as a special consequence, with the additional information that the maximising symbol may be prescribed.
\begin{coro}
\label{label:coro2}
For any admissible $f:\mathbb{N}\to\mathbb{N}$, 
\begin{align*}
    \dim{\left\{x\in[0,1)\setminus\mathbb{Q}:
    \lim_{n\to\infty}\frac{R_n(x)}{f(n)}=1\right\}}=1.
\end{align*}
\end{coro}

The third corollary records the prescribed-symbol dominance consequence of Theorem~\ref{thm:main-sn-final}. Unlike the previous fixed-symbol and overall exact-growth results, this statement has no direct analogue in the earlier literature. Thus this corollary gives a new full-dimensional phenomenon: the partial quotient value responsible for the overall longest run can be fixed in advance.
\begin{coro}
\label{label:coro3}
For any $\lambda\in\mathbb{N}$,
\begin{align*}
    \dim{\left\{x\in[0,1)\setminus\mathbb{Q}:
    \text{for any $n\in\mathbb{N}$ and $\mu\in\mathbb{N}\setminus\{\lambda\}$, }\frac{
    L_n(x,\lambda)}{L_n(x,\mu)}>1\right\}}=1.
\end{align*}
\end{coro}

\section{Basic lemmas}

Let $\mathbb{N}^*\coloneqq\{\emptyset\}\cup\bigcup_{n\in\mathbb{N}}\mathbb{N}^n$ be the finite word space. The standard convention is adopted that juxtaposition of two or more finite words denotes their concatenation. If a finite word is put inside continued-fraction brackets, the brackets denote the finite continued fraction whose entries are the letters of that word. The same convention is used when a finite word is followed by an infinite continued-fraction tail.

Define $\mathcal{I}(\emptyset)\coloneqq[0,1]$ and for any $n\in\mathbb{N}$ and $\omega=(\omega_1,\dots,\omega_n)\in\mathbb{N}^n$, $\mathcal{I}(\omega)$ to be the continued-fraction cylinder of $\omega$:
\begin{align*}
    \mathcal{I}(\omega)
    \coloneqq
    \bigcap_{i=1}^n\overline{\left\{x\in[0,1)\setminus\mathbb{Q}:
    a_i(x)=\omega_i\right\}}.
\end{align*}
Define $K:\mathbb{N}^*\to\mathbb{N}$ to be the continuant function; that is $K(\emptyset)\coloneqq1$ and for any $n\in\mathbb{N}$ and $\omega=(\omega_1,\dots,\omega_n)\in\mathbb{N}^n$, $K(\omega)$ is the denominator of the fraction $[\omega_1,\ldots,\omega_n]$.

\begin{lemma}
\label{lem:continuant-concatenation}
For any $\omega_1,\omega_2\in\mathbb{N}^*$, $K(\omega_1)K(\omega_2)\leq K(\omega_1\omega_2)\leq 2K(\omega_1)K(\omega_2)$.
\end{lemma}
\begin{proof}
The claim is immediate when one of the words is empty; it suffices to handle the non-trivial case. By the continuant identity, one obtains $K(\omega_1\omega_2)=K(\omega_1)K(\omega_2)+K({\omega_{1}}\partial)\,K(\partial\omega_2)$,
where ${\omega_1}\partial$ denotes $\omega_1$ with its last digit removed and $\partial\omega_2$ denotes $\omega_2$ with its first digit removed. The result follows from $0\leq K({\omega_1}\partial)\leq K(\omega_1)$ and $0\leq K(\partial\omega_2)\leq K(\omega_2)$.
\end{proof}

Define, for any $\lambda\in\mathbb{N}$, $\lambda^0\coloneqq\emptyset$ and for $n\in\mathbb{N}$, the long-run word $\lambda^n\coloneqq(\lambda,\lambda,\ldots,\lambda)\in\mathbb{N}^n$.

\begin{lemma}
\label{lem:continuant-growth}
For any $\omega\in\mathbb{N}^n$, $K(\omega)\geq \varphi^{n-1}$. For any $\lambda\in\mathbb{N}$ and $n\in\mathbb{N}\cup\{0\}$, $K(\lambda^n)\leq {\tau_\lambda}^n$.
\end{lemma}
\begin{proof}
For any $n\in\mathbb{N}$ and $\omega\in\mathbb{N}^n$, $K(\omega)\geq K(1^n)\geq \varphi^{n-1}$. One obtains from the standard continuant recurrence formula that for any $n\in\mathbb{N}$, $K(\lambda^{n+1})=\lambda\,K(\lambda^n)+K(\lambda^{n-1})$ and 
\begin{align*}
    K(\lambda^n)
    =\frac{{\tau_\lambda}^{n+1}-(-\tau_\lambda)^{-n-1}}{\sqrt{\lambda^2+4}}
    \leq{\tau_\lambda}^n.
\end{align*}
\end{proof}
Define $T_{\emptyset}\coloneqq\mathrm{id}_{[0,1]}$. Define, for any $n\in\mathbb{N}$ and $\omega=(\omega_1,\ldots,\omega_n)\in\mathbb{N}^n$, the prefix map $T_\omega:[0,1]\to[0,1]$ by, for any $\xi\in[0,1]$,
\begin{align*}
    T_\omega(\xi)\coloneqq \frac{p_n+\xi p_{n-1}}{q_n+\xi q_{n-1}},
\end{align*}
where $p_0\coloneqq0$, $q_0\coloneqq 1$ and for any $k\in\{1,\ldots,n\}$, $p_k$ and $q_k$ are coprime positive integers so that $p_k/q_k=[\omega_1,\ldots,\omega_k]$. 

\begin{lemma}
\label{lem:prefix-map-derivative}
For any $\omega\in\mathbb{N}^*$, one obtains that $T_\omega([0,1])=\mathcal{I}(\omega)$ and for any $\xi\in[0,1]$, 
\begin{align*}
    \frac{1}{4K(\omega)^2}
    \leq\left|T_\omega'(\xi)\right|
    \leq\frac{1}{K(\omega)^2};
\end{align*}
in particular, $T_\omega$ is bi-Lipschitz and preserves the Hausdorff dimension of subsets of $[0,1]$.
\end{lemma}
\begin{proof}
The claim is immediate when the word is empty; it suffices to handle the non-trivial case. Pick any $n\in\mathbb{N}$ and $\omega\in\mathbb{N}^n$. By differentiation and the standard formula $|p_{n-1}q_n-p_nq_{n-1}|=1$, one obtains that for any $\xi\in[0,1]$, $\left|T_\omega'(\xi)\right|=1/(q_n+\xi q_{n-1})^2$. Since $0\leq q_{n-1}\leq q_n=K(\omega)$, one obtains that for any $\xi\in[0,1]$, $K(\omega)\leq q_n+\xi q_{n-1}\leq 2K(\omega)$. Since $T_\omega'$ is bounded by positive constants on $[0,1]$, one obtains from applying the mean value theorem that $T_\omega$ is bi-Lipschitz.
\end{proof}

Define, for any $B\in\mathbb{N}$, 
\begin{align*}
    \mathcal{F}_B
    \coloneqq
    \bigcap_{n\in\mathbb{N}}
    \left\{x\in[0,1)\setminus\mathbb{Q}: a_n(x)\in\{1,\ldots,B\}\right\}.
\end{align*}
\begin{lemma}
\label{lem:bounded-prefix-dimension}
For any $B,n\in\mathbb{N}$ and $\nu\in\{1,\dots,B\}^n$,
\begin{align*}
    \dim{\mathcal{F}_B}
    =\dim{\left(\mathcal{I}(\nu)\cap\mathcal{F}_B\right)}.
\end{align*}
\end{lemma}
\begin{proof}
Pick any $n,B\in\mathbb{N}$ and $\nu\in\{1,\dots,B\}^n$. 
By the definition of the prefix map $T_\nu$, one obtains that for any $\xi\in{\mathcal{F}_B}$, $T_\nu(\xi)=[\nu_1,\ldots,\nu_n,a_1(\xi),a_2(\xi),\ldots]$; hence, $\mathcal{I}(\nu)\cap \mathcal{F}_B=T_\nu(\mathcal{F}_B)$. By Lemma~\ref{lem:prefix-map-derivative}, $T_\nu$ is bi-Lipschitz on $\mathcal{F}_B$ and preserves Hausdorff dimension.
\end{proof}

\begin{lemma}
\label{lem:holder-dimension-estimate}
Let $(X,d_X)$ and $(Y,d_Y)$ be metric spaces. Let $g:X\to Y$ be a surjective function. Suppose there exist $0<\alpha\leq 1$ and $C>0$ such that for any $x_1,x_2\in X$,
\begin{align*}
    d_Y(g(x_1),g(x_2))\leq C \left(d_X(x_1,x_2)\right)^\alpha.
\end{align*}
Then $\dim{Y}\leq\dim{X}/\alpha$.
\end{lemma}
\begin{proof}
Pick any $s>\dim{X}$ and $\varepsilon>0$. Since $\mathcal{H}^s(X)=0$, there exists an open cover $\mathcal{U}_{\varepsilon}$ of $X$ such that $\sum_{U\in\mathcal{U}_{\varepsilon}}(\diam{U})^s<\varepsilon$ and for any $U\in\mathcal{U}_{\varepsilon}$, $\diam{U}<\varepsilon$. Since $g$ is surjective, $(g(U))_{U\in\mathcal{U}_{\varepsilon}}$ covers $Y$. Then for any $U\in\mathcal{U}_{\varepsilon}$, $\diam{g(U)}\leq C\,(\diam{U})^\alpha$, and
\begin{align*}
    \sum_{U\in\mathcal{U}_{\varepsilon}}(\diam{g(U)})^{s/\alpha}
    \leq C^{s/\alpha}\sum_{U\in\mathcal{U}_{\varepsilon}} (\diam{U})^s
    < C^{s/\alpha}\varepsilon.
\end{align*}
Hence, for any $s>\dim{X}$, $\mathcal{H}^{s/\alpha}(Y)=0$ and $\dim{Y}\leq s/\alpha$.
\end{proof}

The next lemma is a reformulation of~\cite[Lemma~2.5]{WangWu2011} in the present functional notation.
\begin{lemma}
\label{lem:fixed-multiple-self-neglecting}
Let $f:\mathbb{N}\to\mathbb{N}$. 
Suppose $f$ is admissible. 
Then for any $C\in\mathbb{N}$,
\begin{align*}
    \lim_{n\to\infty}\frac{f(n+C f(n))}{f(n)}
    =1.
\end{align*}
\end{lemma}

\section{Construction of a Cantor-type set}

Let $\lambda\in\mathbb{N}$ and $f:\mathbb{N}\to\mathbb{N}$ be a function. Define $H_\lambda\coloneqq\lfloor{\log{\tau_{\lambda}}}/{\log{\varphi}}\rfloor+1\in\mathbb{N}$. Pick any $M\in\mathbb{N}$. Define $t_0\coloneqq 1$ and for any $k\in\mathbb{N}$,
\begin{align*}
    r_k\coloneqq f(t_{k-1}),\qquad
    m_k\coloneqq \left\lfloor\sqrt{r_k}\right\rfloor, \qquad
    g_k\coloneqq MH_\lambda r_k, \qquad
    p_k\coloneqq \left\lfloor \frac{g_k}{m_k}\right\rfloor, \qquad
    q_k\coloneqq g_k-p_km_k,
\end{align*}
and
\begin{align*}
    s_k\coloneqq f(t_{k-1}+g_k), \qquad 
    d_k\coloneqq 2p_k+2+s_k, \qquad 
    l_k\coloneqq g_k+d_k, \qquad 
    t_k\coloneqq t_{k-1}+l_k.
\end{align*}

Let $k\in\mathbb{N}$. The quantity $r_k$ is the target scale at the beginning of stage $k$, while $s_k$ is the target scale associated with the end of the retained free length $g_k$. The length $g_k$ specifies the number of free digits retained at stage $k$. The auxiliary length $m_k$ tends to infinity while satisfying $\lim_{k\to\infty}m_k/r_k=0$; it will be used to subdivide the free digits into shorter pieces, so that accidental runs remain negligible compared with the target scale. The integers $p_k$ and $q_k$ are the quotient and remainder obtained when the free length $g_k$ is divided into pieces of length $m_k$. Finally, $d_k$, $l_k$, and $t_k$ record respectively the non-free length to be added at stage $k$, the total length of stage $k$, and the starting position of stage $k+1$.

\begin{prop}
\label{prop:scale-inputs-final}
Let $\lambda\in\mathbb{N}$ and $f:\mathbb{N}\to\mathbb{N}$ be a function. 
Suppose $f$ is admissible. Then for any $M\in\mathbb{N}$,
\begin{align}
\label{eq:scale-inputs-final}
    \lim_{k\to\infty}\frac{s_k}{r_k}
    =\lim_{k\to\infty}\frac{s_{k-1}}{r_k}
    =1,
    \qquad
    \lim_{k\to\infty}\frac{p_k}{r_k}=0,
\end{align}
and there exists $K_{\lambda,f,M}\in\mathbb{N}$ such that for any $k\in\mathbb{N}$ if $k\geq K_{\lambda,f,M}$ then
\begin{align}
\label{eq:ell-eventual-bound-final}
    l_k\leq (MH_\lambda+2)\,r_k.
\end{align}
\end{prop}
\begin{proof}
Since $f$ is unbounded and non-decreasing, one obtains that $(t_k)_{k\in\mathbb{N}}$ is strictly increasing and $\lim_{k\to\infty}r_k=\lim_{k\to\infty}f(t_{k-1})=+\infty$. By Lemma~\ref{lem:fixed-multiple-self-neglecting}, one obtains
\begin{align*}
    \lim_{k\to\infty}\frac{f(t_{k-1}+MH_\lambda f(t_{k-1}))}{f(t_{k-1})}=1;
\end{align*}
hence, $\lim_{k\to\infty}{s_k}/{r_k}=1$. Since $\lim_{k\to\infty}r_k=+\infty$, one obtains $\lim_{k\to\infty}m_k=+\infty$.

Since for any $k\in\mathbb{N}$, $p_km_k\leq g_k=MH_\lambda r_k$, one obtains $0\leq {p_k}/{r_k}\leq {MH_\lambda}/{m_k}$; hence, $\lim_{k\to\infty}{p_k}/{r_k}=0$. Note that for any $k\in\mathbb{N}$,
\begin{align*}
    \frac{l_k}{r_k}
    =
    \frac{g_k+d_k}{r_k}
    =
    MH_\lambda+\frac{2p_k}{r_k}+\frac{2}{r_k}+\frac{s_k}{r_k};
\end{align*}
hence, $\lim_{k\to\infty}{l_k}/{r_k}=MH_\lambda+1$ and there exists $K_{\lambda,f,M}\in\mathbb{N}$ such that for any $k\in\mathbb{N}$ if $k\geq K_{\lambda,f,M}$ then~\eqref{eq:ell-eventual-bound-final} holds.

One obtains that for any $k\in\mathbb{N}$ if $k\geq K_{\lambda,f,M}+1$ then 
\begin{align*}
    t_{k-1}
    =t_{k-2}+l_{k-1}
    \leq t_{k-2}+(MH_\lambda+2)\,r_{k-1}
    =t_{k-2}+(MH_\lambda+2)\,f(t_{k-2}),
\end{align*}
and
\begin{align*}
    1\leq \frac{f(t_{k-1})}{f(t_{k-2})}\le
    \frac{f(t_{k-2}+(MH_\lambda+2)\,f(t_{k-2}))}{f(t_{k-2})}.
\end{align*}
By Lemma~\ref{lem:fixed-multiple-self-neglecting}, one obtains $\lim_{k\to\infty}{f(t_{k-1})}/{f(t_{k-2})}=1$; hence, $\lim_{k\to\infty}{r_{k-1}}/{r_k}=1$. Thus,
\begin{align*}
    \lim_{k\to\infty}\frac{s_{k-1}}{r_k}
    =\lim_{k\to\infty}\frac{s_{k-1}}{r_{k-1}}\lim_{k\to\infty}\frac{r_{k-1}}{r_k}
    =1.
\end{align*}
\end{proof}

Let $\lambda\in\mathbb{N}$ and $f:\mathbb{N}\to\mathbb{N}$ be a function. Pick any $M,B\in\mathbb{N}$. Suppose $B\geq\lambda+2$.
Define, for any $k\in\mathbb N$ and $U_k\in\{1,\dots,B\}^{g_k}$, the expanded block $U^+_k\in\{1,\dots,B\}^{l_k}$ by:
\begin{align*}
    U^+_k
    \coloneqq
    U_{k,1}\sigma_1\sigma_2U_{k,2}\sigma_1\sigma_2\cdots U_{k,p_k}\sigma_1\sigma_2U_{k,p_k+1}\sigma_1\lambda^{s_k}\sigma_2,
\end{align*}
where $\sigma_1\coloneqq \lambda+1$ and $\sigma_2\coloneqq \lambda+2$ are the separator digits, $U_{k,1},\dots,U_{k,p_k}\in \{1,\dots,B\}^{m_k}$ and $U_{k,p_k+1}\in \{1,\dots,B\}^{q_k}$ are the unique words satisfying 
\begin{align*}
    U_k=U_{k,1}U_{k,2}\cdots U_{k,p_k}U_{k,p_k+1},
\end{align*}
with the convention that $\{1,\ldots,B\}^0\coloneqq \{\emptyset\}$. Define
\begin{align*}
    \mathcal{K}_{\lambda,f,M,B}
    \coloneqq\left\{\left[U^+_1U^+_2U^+_3\cdots\right]:
    U_k\in\{1,\dots,B\}^{g_k}\text{ for all $k\in\mathbb{N}$.}\right\}
    \subset\mathcal{F}_B.
\end{align*}
Define $\zeta_{\lambda,f,M,B}:\mathcal{F}_B\to\mathcal{K}_{\lambda,f,M,B}$ by for any $y\in \mathcal{F}_B$,
\begin{align*}
    \zeta_{\lambda,f,M,B}(y)
    \coloneqq\left[U^+_{y,1}U^+_{y,2}U^+_{y,3}\cdots\right],
\end{align*}
where the partial quotients of $y$ are partitioned into consecutive words $U_{y,1},U_{y,2},U_{y,3},\ldots$ so that $y=[U_{y,1}U_{y,2}U_{y,3}\cdots]$ and for any $k\in\mathbb{N}$, $U_{y,k}\in\{1,\dots,B\}^{g_k}$. 

Define
\begin{align*}
    \mathfrak{F}_{\lambda,f}
    &\coloneqq
    \left\{x\in[0,1)\setminus\mathbb{Q}:
    \lim_{n\to\infty}\frac{L_n(x,\lambda)}{f(n)}=1
    \text{ and }\right.\\
    &\qquad\qquad\left.\text{for all sufficiently large $n\in\mathbb{N}$ and for any $\mu\in\mathbb{N}\setminus\{\lambda\}$, }\frac{L_n(x,\lambda)}{L_n(x,\mu)}>1\right\}.
\end{align*}

\begin{prop}
\label{prop:run-control-final}
Let $\lambda\in\mathbb{N}$ and $f:\mathbb{N}\to\mathbb{N}$ be a function. Suppose $f$ is admissible. Then for any $M,B\in\mathbb{N}$, if $B\geq \lambda+2$ then
\begin{align*}
    \mathcal{K}_{\lambda,f,M,B}\subset \mathfrak{F}_{\lambda,f}.
\end{align*}
\end{prop}
\begin{proof}
Pick any $x\in \mathcal{K}_{\lambda,f,M,B}$. Since $\mathcal{K}_{\lambda,f,M,B}=\zeta_{\lambda,f,M,B}(\mathcal{F}_B)$, there exists $y\in \mathcal{F}_B$ such that $x=\zeta_{\lambda,f,M,B}(y)$. For any $k\in\mathbb{N}$, the expanded block $U^+_{y,k}$ has the following form
\begin{align*}
    U^+_{y,k}
    =U_{y,k,1}\sigma_1\sigma_2U_{y,k,2}\sigma_1\sigma_2\cdots U_{y,k,p_k}\sigma_1\sigma_2U_{y,k,p_k+1}\sigma_1\lambda^{s_k}\sigma_2,
\end{align*}
where $U_{y,k,1},\dots,U_{y,k,p_k}\in \{1,\dots,B\}^{m_k}$ and $U_{y,k,p_k+1}\in \{1,\dots,B\}^{q_k}$ are the unique words satisfying $U_{y,k}=U_{y,k,1}U_{y,k,2}\cdots U_{y,k,p_k}U_{y,k,p_k+1}$. Pick any $k\in\mathbb{N}$. Every consecutive run in the first $t_k-1$ digits of $x$ which is not contained in one of the prescribed words $\lambda^{s_1},\ldots,\lambda^{s_k}$ lies either inside one of the words $U_{y,j,i}$ for some $j\in\{1,\ldots,k\}$ and $i\in\{1,\ldots,p_j+1\}$, or across a boundary of one of the forms
\begin{align*}
    U_{y,j,i}\sigma_1,\qquad
    \sigma_2U_{y,j,i+1},\qquad
    U_{y,j,p_j+1}\sigma_1,\qquad
    \sigma_2U_{y,\min\{j+1,k\},1},
\end{align*}
for some $j\in\{1,\ldots,k\}$ and $i\in\{1,\ldots,p_j\}$. Since $f$ is non-decreasing, the sequence $(m_k)_{k\in\mathbb{N}}$ is non-decreasing. Hence, each such consecutive run has length at most $m_k+1$. 

Since $\lim_{k\to\infty}r_k=\lim_{k\to\infty}m_k=+\infty$ and for any $k\in\mathbb{N}$, 
\begin{align*}
    0
    \leq \frac{m_k+1}{r_k}=\frac{m_k}{r_k}+\frac{1}{r_k}
    \leq \frac{1}{m_k}+\frac{1}{r_k},
\end{align*}
one obtains $\lim_{k\to\infty}{(m_k+1)}/{r_k}=0$. By~\eqref{eq:scale-inputs-final}, one obtains $\lim_{k\to\infty}{(m_k+1)}/{s_{k-1}}=0$ and that there exists $k_0\in\mathbb{N}$ such that $k_0\geq2$ and for any $k\in\mathbb{N}$, if $k\geq k_0$ then $m_k+1<s_{k-1}$. 

One obtains that for any $k\in\mathbb{N}$, $t_k=t_{k-1}+l_k>t_{k-1}+g_k$; hence, the sequence $(t_{k-1}+g_k)_{k\in\mathbb{N}}$ is strictly increasing. Since $f$ is non-decreasing, the sequence $(s_k)_{k\in\mathbb{N}}$ is unbounded and non-decreasing. One obtains that there exists $k_1\in\mathbb{N}$ such that $k_1\geq k_0+1$ and for any $k\in\mathbb{N}$, if $k\geq k_1$ then
\begin{align*}
    C_0\coloneqq\max{\left\{s_{k_0-1},m_{k_0-1}+1\right\}}<s_{k-1}.
\end{align*}

Pick any $k,n\in\mathbb{N}$. Suppose $k\geq k_1$ and $t_{k-1}\leq n< t_k$. Since the first $n$ partial quotients of $x$ contain the full block $\lambda^{s_{k-1}}$, one obtains $L_n(x,\lambda)\geq s_{k-1}$. On the other hand, every $\lambda$-run outside the prescribed blocks has length at most
\begin{align*}
    \max{\left\{C_0,m_k+1\right\}}<s_{k-1}\leq s_k.
\end{align*}
Every prescribed block from an earlier stage has length at most $s_{k-1}$, and the part of the current prescribed block already contained in the first $n$ partial quotients of $x$ has length at most $s_k$. Therefore, for any $k,n\in\mathbb{N}$, if $k\geq k_1$ and $t_{k-1}\leq n< t_k$ then
\begin{align}
    \label{eq:Ln-sandwich-final}
    s_{k-1}
    \leq L_n(x,\lambda)
    \leq s_k.
\end{align}

By Proposition~\ref{prop:scale-inputs-final}, there exists $K_{\lambda,f,M}\in\mathbb{N}$ such that for any $k\in\mathbb{N}$, if $k\geq K_{\lambda,f,M}$ then 
\begin{align*}
    t_k=t_{k-1}+l_k
    \leq t_{k-1}+(MH_\lambda+2)\,r_k
    =t_{k-1}+(MH_\lambda+2)\, f(t_{k-1}).
\end{align*}
One obtains that for any $k,n\in\mathbb{N}$, if $k\geq K_{\lambda,f,M}$ and $t_{k-1}\leq n< t_k$ then
\begin{align*}
    t_{k-1}\leq n< t_k\leq t_{k-1}+(MH_\lambda+2)\,f(t_{k-1});
\end{align*}
hence, one obtains from $f$ being non-decreasing that
\begin{align*}
    f(t_{k-1})
    \leq f(n)
    \leq f(t_{k-1}+(MH_\lambda+2)\,f(t_{k-1})).
\end{align*}
One obtains from Lemma~\ref{lem:fixed-multiple-self-neglecting} that
\begin{align*}
    \lim_{k\to\infty}\frac{f(t_{k-1}+(MH_\lambda+2)\,f(t_{k-1}))}{f(t_{k-1})}
    =1.
\end{align*}
By the definition of $(r_k)_{k\in\mathbb{N}}$ and the squeeze theorem, one obtains
\begin{align*}
    \lim_{k\to\infty}\sup_{n\in\{t_{k-1},\ldots,t_k-1\}}
    \left|\frac{f(n)}{r_k}-1\right|
    =0.
\end{align*}
Therefore, one obtains from combining~\eqref{eq:scale-inputs-final} and~\eqref{eq:Ln-sandwich-final} that
\begin{align*}
    \lim_{n\to\infty}\frac{L_n(x,\lambda)}{f(n)}=1.
\end{align*}

One remains to prove eventual dominance. Pick any $k,n\in\mathbb{N}$ and $\mu\in\mathbb{N}\setminus\{\lambda\}$. Suppose $k\geq k_1$ and $t_{k-1}\leq n< t_k$. Every $\mu$-run in the first $n$ partial quotients of $x$ is necessarily not one of the prescribed $\lambda$-blocks. Since $\mu\neq\lambda$, no $\mu$-run can be contained in a prescribed block $\lambda^{s_j}$. Hence, its length is bounded by the maximal length of a non-prescribed run, and one obtains
\begin{align*}
    L_n(x,\mu)\leq\max{\{C_0,m_k+1\}}<s_{k-1}\leq L_n(x,\lambda).
\end{align*}
Therefore, one obtains that for any $n\in\mathbb N$ and $\mu\in\mathbb{N}\setminus\{\lambda\}$, if $n\geq N_{\lambda,f,M}\coloneqq t_{k_1-1}$, then
\begin{align*}
    L_n(x,\lambda)>L_n(x,\mu).
\end{align*}
\end{proof}

\section{Continuant distortion}

Let $\lambda\in\mathbb{N}$. Define $C_\lambda\coloneqq8(\lambda+1)(\lambda+2)$. Let $M\in\mathbb{N}$ and $0<\alpha<M/(M+1)$. Note that $(1-\alpha)M>\alpha$ and $H_\lambda\log{\varphi}>\log{\tau_\lambda}$. Define
\begin{align*}
    \eta_{\lambda,M,\alpha}
    \coloneqq\frac{(1-\alpha)MH_\lambda\log\varphi-\alpha\log \tau_\lambda}{4}>0.
\end{align*}

\begin{lemma}
\label{lem:full-stage-comparison}
Let $\lambda\in\mathbb{N}$ and $f:\mathbb{N}\to\mathbb{N}$ be a function. Suppose $f$ is admissible. Then for any $M,B\in\mathbb{N}$ and $0<\alpha<M/(M+1)$, there exists $k_{\lambda,f,M,\alpha,1}\in\mathbb{N}$ such that for any $k\in\mathbb{N}$, if $k\geq k_{\lambda,f,M,\alpha,1}$ then for any $U_k\in\{1,\ldots,B\}^{g_k}$,
\begin{align*}
    K(U^+_k)^\alpha
    \leq 2^{-2\alpha}K(U_k).
\end{align*}
\end{lemma}
\begin{proof}
Pick any $k\in\mathbb{N}$. By Lemma~\ref{lem:continuant-concatenation}, one obtains
\begin{align*}
    K(U^+_k)
    \leq
    {C_\lambda}^{p_k+1}K(\lambda^{s_k})K(U_k).
\end{align*}
By Lemma~\ref{lem:continuant-growth}, one obtains
\begin{align*}
    \log{\frac{K(U^+_k)^\alpha}{K(U_k)}}
    \leq
    \alpha(p_k+1)\log C_\lambda
    +\alpha s_k\log\tau_\lambda
    -(1-\alpha)(g_k-1)\log\varphi.
\end{align*}
By the definition of $(g_k)_{k\in\mathbb{N}}$, one obtains
\begin{align*}
    \log{\frac{K(U^+_k)^\alpha}{K(U_k)}}
    \leq
    \alpha(p_k+1)\log C_\lambda+\alpha(s_k-r_k)\log\tau_\lambda
    -4\eta_{\lambda,M,\alpha} r_k+(1-\alpha)\log{\varphi}.
\end{align*}
By Proposition~\ref{prop:scale-inputs-final} and $\lim_{k\to\infty}r_k=+\infty$, one obtains that for all sufficiently large $k\in\mathbb{N}$, 
\begin{align*}
    \log{\frac{K(U^+_k)^\alpha}{K(U_k)}}\leq-2\eta_{\lambda,M,\alpha} r_k\leq -2\alpha\log 2.
\end{align*}
The result follows after exponentiating.
\end{proof}

\begin{lemma}
\label{lem:stage-prefix-comparison}
Let $\lambda\in\mathbb{N}$ and $f:\mathbb{N}\to\mathbb{N}$ be a function. Suppose $f$ is admissible. Then for any $M,B\in\mathbb{N}$ and $0<\alpha<M/(M+1)$, there exist $C_{\lambda,f,M,\alpha,2}>0$ and $k_{\lambda,f,M,\alpha,2}\in\mathbb N$ such that for any $k\in\mathbb{N}$ and $U_k\in\{1,\ldots,B\}^{g_k}$, if $k\geq k_{\lambda,f,M,\alpha,2}$ then for any prefix $W$ of $U^+_k$,
\begin{align*}
    K(W)^\alpha\leq C_{\lambda,f,M,\alpha,2}\,K(W^-),
\end{align*}
where $W^-$ is the underlying word obtained from $W$ by deleting all inserted symbols $\sigma_1,\sigma_2$, as specified by their positions in the construction, not all occurrences of the numerical values $\lambda+1$ or $\lambda+2$, and any terminal segment of the form $\sigma_1\lambda^u \sigma_2$ or $\sigma_1\lambda^u$ or $\sigma_1$ coming from the final inserted block.
\end{lemma}
\begin{proof}
Pick any $k\in\mathbb{N}$ and $U_k\in\{1,\ldots,B\}^{g_k}$. Let $W$ be a prefix of $U^+_k$.

Suppose $W$ does not enter the terminal block $\sigma_1\lambda^{s_k}\sigma_2$. Then $W$ consists of $j$ complete pieces $U_{k,i}\sigma_1\sigma_2$, followed by a final prefix lying in one of $U_{k,j+1}$, $U_{k,j+1}\sigma_1$, or $U_{k,j+1}\sigma_1\sigma_2$, with $0\leq j\leq p_k$. The corresponding retained word $W^-$ is a prefix of $U_k$, and $W^-$ is of length at least $jm_k$. By repeated use of Lemma~\ref{lem:continuant-concatenation}, each insertion of a block $\sigma_1\sigma_2$ between two retained subwords increases the continuant by at most a multiplicative factor $4K(\sigma_1\sigma_2)$, while appending a
terminal $\sigma_1$ or $\sigma_1\sigma_2$ contributes at most a factor $8K(\sigma_1)K(\sigma_2)$. Hence, for any $j\in\{0,\ldots,p_k\}$,
\begin{align*}
    K(W)\leq {C_\lambda}^{j+1}K(W^-).
\end{align*}
By Lemma~\ref{lem:continuant-growth}, one obtains that for any $j\in\{0,\ldots,p_k\}$,
\begin{align*}
    \log\frac{K(W)^\alpha}{K(W^-)}
    &\leq
    \alpha(j+1)\log C_\lambda
    -(1-\alpha)\log K(W^-) \\
    &\leq
    \alpha(j+1)\log C_\lambda-(1-\alpha)(j m_k-1)\log\varphi.
\end{align*}
Since $\lim_{k\to\infty}m_k=+\infty$, the coefficient of $j$ on the right-hand side is negative for all sufficiently large $k\in\mathbb{N}$. Hence, the right-hand side is bounded above by a constant depending only on $\lambda$, $M$ and $\alpha$.

Suppose $W$ enters the terminal block $\sigma_1\lambda^{s_k}\sigma_2$. In this case, there exist $u\in\{0,\ldots,s_k\}$ and $\theta\in\{\varnothing,\sigma_2\}$ such that
\begin{align*}
    W=
    U_{k,1}\sigma_1\sigma_2\cdots U_{k,p_k}\sigma_1\sigma_2U_{k,p_k+1}\sigma_1\lambda^{u}\theta;
\end{align*}
in particular $W^-=U_k$. By Lemma~\ref{lem:continuant-concatenation},
\begin{align*}
    K(W)\leq {C_\lambda}^{p_k+1}K(\lambda^u)K(U_k).
\end{align*}
By Lemma~\ref{lem:continuant-growth},
\begin{align*}
    \log\frac{K(W)^\alpha}{K(W^-)}
    =
    \log\frac{K(W)^\alpha}{K(U_k)}
    \leq
    \alpha(p_k+1)\log C_\lambda+\alpha(u+1)\log\tau_\lambda-(1-\alpha)(g_k-1)\log\varphi.
\end{align*}
By $u\leq s_k$, $g_k=MH_\lambda r_k$ and Proposition~\ref{prop:scale-inputs-final}, one obtains for all sufficiently large $k\in\mathbb{N}$, 
\begin{align*}
    \log\frac{K(W)^\alpha}{K(W^-)}
    \leq -2\eta_{\lambda,M,\alpha} r_k<0.
\end{align*}
\end{proof}

Let $\lambda,M,B\in\mathbb{N}$ and $f:\mathbb{N}\to\mathbb{N}$ be a function. Suppose $B\geq\lambda+2$. Define, for any $k\in\mathbb{N}$, $\mathscr{C}_k(\mathcal{K}_{\lambda,f,M,B})$ to be the family of stage-$k$ cylinders of $\mathcal{K}_{\lambda,f,M,B}$; that is,
\begin{align*}
    \mathscr{C}_k(\mathcal{K}_{\lambda,f,M,B})
    \coloneqq
    \left\{J=\mathcal{I}\!\left(U^+_1U^+_2\cdots U^+_k\right)\cap\mathcal{K}_{\lambda,f,M,B}:
    U_i\in\{1,\dots,B\}^{g_i}\text{ for all $i\in\{1,2,\ldots,k\}$.}\right\}.
\end{align*}
Define $\pi_{\lambda,f,M,B}:\mathcal{K}_{\lambda,f,M,B}\to \mathcal{F}_B$ to be the map obtained by deleting, in each stage, all inserted separators $\sigma_1\sigma_2$ together with the terminal block $\sigma_1\lambda^{s_k}\sigma_2$. Note that $\pi_{\lambda,f,M,B}$ is surjective and $\pi_{\lambda,f,M,B}\circ\zeta_{\lambda,f,M,B}=\mathrm{id}_{\mathcal{F}_B}$.

\begin{prop}
\label{prop:local-holder-final}
Let $\lambda\in\mathbb{N}$ and $f:\mathbb{N}\to\mathbb{N}$ be a function. Suppose $f$ is admissible. Then for any $M,B\in\mathbb{N}$ and $0<\alpha<M/(M+1)$, if $B\geq\lambda+2$ then there exist $k_{\lambda,f,M,\alpha}\in\mathbb{N}$ and $C_{\lambda,f,M,B,\alpha}>0$ such that for any
\begin{align*}
    J
    =\mathcal{I}\!\left(U^+_1U^+_2\cdots U^+_{k_{\lambda,f,M,\alpha}}\right)\cap\mathcal{K}_{\lambda,f,M,B}
    \in\mathscr{C}_{k_{\lambda,f,M,\alpha}}(\mathcal{K}_{\lambda,f,M,B}),
\end{align*}
where \(U_1,\ldots,U_{k_{\lambda,f,M,\alpha}}\) are the unique words determining \(J\), the restriction of $\pi_{\lambda,f,M,B}$ to $J$:
\begin{align*}
    \pi_{\lambda,f,M,B}|_J:J\to \mathcal{I}\!\left(U_1U_2\cdots U_{k_{\lambda,f,M,\alpha}}\right)\cap\mathcal{F}_B,
\end{align*}
is surjective and
for any $x,y\in J$,
\begin{align*}
    \left|\pi_{\lambda,f,M,B}(x)-\pi_{\lambda,f,M,B}(y)\right|
    \leq C_{\lambda,f,M,B,\alpha} \left|x-y\right|^\alpha.
\end{align*}
\end{prop}
\begin{proof}
The surjectivity of $\pi_{\lambda,f,M,B}|_J$ is immediate from the definition of $\pi$ and $\zeta_{\lambda,f,M,B}$.

Define $k_{\lambda,f,M,\alpha}\coloneqq \max{\{k_{\lambda,f,M,\alpha,1},k_{\lambda,f,M,\alpha,2}\}}$, where $k_{\lambda,f,M,\alpha,1}$ and $k_{\lambda,f,M,\alpha,2}$ are given by Lemmas~\ref{lem:full-stage-comparison} and~\ref{lem:stage-prefix-comparison} respectively.
Pick any $J\in\mathscr{C}_{k_{\lambda,f,M,\alpha}}(\mathcal{K}_{\lambda,f,M,B})$. There exist unique $U_{J,i}\in\{1,\ldots,B\}^{g_i}$ such that 
\begin{align*}
    J=\mathcal{I}(V^+_J)\cap\mathcal{K}_{\lambda,f,M,B},
\end{align*}
where $V^+_J\coloneqq U_{J,1}^+U_{J,2}^+\cdots U_{J,k_{\lambda,f,M,\alpha}}^+$. Define $V_J\coloneqq U_{J,1}U_{J,2}\cdots U_{J,k_{\lambda,f,M,\alpha}}$ and
\begin{align*}
    C_{\lambda,f,M,B,\alpha,0}
    \coloneqq 2^{2\alpha}C_{\lambda,f,M,\alpha,2}
    \max_{J\in\mathscr{C}_{k_{\lambda,f,M,\alpha}}(\mathcal{K}_{\lambda,f,M,B})}{\frac{K(V^+_J)^\alpha}{K(V_J)}}<+\infty,
\end{align*}
as the family $\mathscr{C}_{k_{\lambda,f,M,\alpha}}(\mathcal{K}_{\lambda,f,M,B})$ is finite. Define, for any $u\in\{1,\dots,B\}$, the interval $J_{B,u}$ by
\begin{align*}
    J_{B,u}\coloneqq
    \left[
    \frac{1}{u+1/\varphi},
    \frac{1}{u+1/\tau_B}
    \right].
\end{align*}
Note that for any $\xi\in\mathcal{F}_B$ and $u\in\{1,\ldots,B\}$, if $a_1(\xi)=u$ then $T(\xi)\in[1/\tau_B,1/\varphi]$ and $\xi\in J_{B,u}$. Thus, the family $\mathscr{J}_B\coloneqq(J_{B,u})_{u\in\{1,\ldots,B\}}$ consists of pairwise disjoint compact intervals, and one obtains:
\begin{align*}
    \Delta_B
    \coloneqq
    \min_{1\leq u_1<u_2\leq B}\operatorname{dist}{\left(J_{B,u_1},J_{B,u_2}\right)}>0.
\end{align*}

Pick any $x,y\in J$. Suppose $x\ne y$. Define $\omega$ to be the longest common word after the fixed prefix $V^+_J$; that is, there exist $m\geq k_{\lambda,f,M,\alpha}+1$, words $U_j\in\{1,\ldots,B\}^{g_j}$ for $j\in\{k_{\lambda,f,M,\alpha}+1,\ldots,m-1\}$, a word $U_m\in\{1,\ldots,B\}^{g_m}$ and a prefix $W$ of $U_m^+$ such that
\begin{align*}
    \omega=U^+_{k_{\lambda,f,M,\alpha}+1}\cdots U^+_{m-1}W
\end{align*}
(the case $m=k_{\lambda,f,M,\alpha}+1$ and $\omega=W$ is allowed). Note that
\begin{align*}
    x=T_{V^+_J\omega}(\xi_x), \qquad
    y=T_{V^+_J\omega}(\xi_y),
\end{align*}
and $a_1(\xi_x)\ne a_1(\xi_y)$. Since $\mathcal{K}_{\lambda,f,M,B}\subset\mathcal{F}_B$, one obtains $a_1(\xi_x),a_1(\xi_y)\in\{1,\dots,B\}$ and $\xi_x,\xi_y\in\mathcal{F}_B\subset[1/\tau_B,1/\varphi]$. Hence, $\xi_x$ and $\xi_y$ belong to two distinct intervals among $\mathscr{J}_B$, so
\begin{align*}
    |\xi_x-\xi_y|\geq \Delta_B.
\end{align*}

By Lemma~\ref{lem:prefix-map-derivative} and the mean value theorem, one obtains
\begin{align}
\label{eq:x-y-lower-local-holder-revised}
    \left|x-y\right|
    =
    \left|T_{V^+_J\omega}(\xi_x)-T_{V^+_J\omega}(\xi_y)\right|
    \ge
    \frac{\Delta_B}{4K(V^+_J\omega)^2}.
\end{align}
Let $\omega^-$ be the word obtained from $\omega$ by deleting all
inserted symbols stage by stage; that is,
\begin{align*}
    \omega^-=U_{k_{\lambda,f,M,\alpha}+1}\cdots U_{m-1}W^-.
\end{align*}
By the definition of $\pi_{\lambda,f,M,B}$ and $W^-$, deleting the inserted symbols from the common prefix $V_J^+\omega$ yields exactly $V_J\omega^-$. Then $\pi_{\lambda,f,M,B}(x)$ and $\pi_{\lambda,f,M,B}(y)$ agree on the prefix $V_J\omega^-$. By the standard cylinder diameter estimate, one obtains 
\begin{align}
\label{eq:pi-x-y-upper-local-holder-revised}
    \left|\pi_{\lambda,f,M,B}(x)-\pi_{\lambda,f,M,B}(y)\right|
    \leq \operatorname{diam}{\mathcal{I}\!\left(V_J\omega^-\right)}
    \leq \frac{1}{K(V_J\omega^-)^2}.
\end{align}

By repeated application of Lemma~\ref{lem:continuant-concatenation}, one obtains
\begin{align*}
    K(V^+_J\omega)
    \leq
    2^{m-k_{\lambda,f,M,\alpha}}
    K(V^+_J)
    \prod_{j=k_{\lambda,f,M,\alpha}+1}^{m-1}K(U^+_j)\,K(W),
\end{align*}
where the empty product is understood to be equal to $1$, and
\begin{align*}
    K(V_J\omega^-)
    \geq
    K(V_J)
    \prod_{j=k_{\lambda,f,M,\alpha}+1}^{m-1}K(U_j)\,K(W^-).
\end{align*}
By applying Lemma~\ref{lem:full-stage-comparison} to the full stages and Lemma~\ref{lem:stage-prefix-comparison} to $W$ respectively, one obtains
\begin{align*}
    K(V^+_J\omega)^\alpha
    &\leq
    2^{(m-k_{\lambda,f,M,\alpha})\alpha}
    K(V^+_J)^\alpha
    \prod_{j=k_{\lambda,f,M,\alpha}+1}^{m-1}K(U^+_j)^\alpha
    K(W)^\alpha \\
    &\leq
    2^{(m-k_{\lambda,f,M,\alpha})\alpha}
    K(V^+_J)^\alpha
    \left(\prod_{j=k_{\lambda,f,M,\alpha}+1}^{m-1}2^{-2\alpha}K(U_j)\right)
    C_{\lambda,f,M,\alpha,2}K(W^-).
\end{align*}
One obtains from $2^{(m-k_{\lambda,f,M,\alpha})\alpha}\prod_{j=k_{\lambda,f,M,\alpha}+1}^{m-1}2^{-2\alpha}=2^{(2+k_{\lambda,f,M,\alpha}-m)\alpha}\leq 2^{2\alpha}$ and the definition of $C_{\lambda,f,M,B,\alpha,0}$ that
\begin{align*}
    K(V^+_J\omega)^\alpha
    \leq
    \frac{2^{2\alpha}C_{\lambda,f,M,\alpha,2}K(V^+_J)^\alpha}{K(V_J)}
    K(V_J\omega^-)
    \leq
    C_{\lambda,f,M,B,\alpha,0}K(V_J\omega^-).
\end{align*}
By combining the above with~\eqref{eq:x-y-lower-local-holder-revised} and~\eqref{eq:pi-x-y-upper-local-holder-revised}, one obtains
\begin{align*}
    \left|\pi_{\lambda,f,M,B}(x)-\pi_{\lambda,f,M,B}(y)\right|
    &\leq \frac{1}{K(V_J\omega^-)^2} 
    \leq \frac{(C_{\lambda,f,M,B,\alpha,0})^2}{K(V^+_J\omega)^{2\alpha}}
    \leq (C_{\lambda,f,M,B,\alpha,0})^2\left(\frac{4}{\Delta_B}\right)^\alpha \left|x-y\right|^\alpha \\
    &\leq C_{\lambda,f,M,B,\alpha}\left|x-y\right|^\alpha,
\end{align*}
where $C_{\lambda,f,M,B,\alpha}\coloneqq (C_{\lambda,f,M,B,\alpha,0})^2({4}/{\Delta_B})^\alpha$.
\end{proof}

\section{Dimension estimate}

\begin{prop}
\label{prop:dimension-lower-final}
Let $\lambda\in\mathbb{N}$ and $f:\mathbb{N}\to\mathbb{N}$ be a function. Suppose $f$ is admissible. Then for any $M,B\in\mathbb{N}$ and $0<\alpha<M/(M+1)$, if $B\geq \lambda+2$ then
\begin{align*}
    \dim{\mathcal{K}_{\lambda,f,M,B}}
    \geq \alpha\dim{\mathcal{F}_B}.
\end{align*}
\end{prop}
\begin{proof}
Let $k_{\lambda,f,M,\alpha}$ be given by Proposition~\ref{prop:local-holder-final}. Since the family $\mathscr{C}_{k_{\lambda,f,M,\alpha}}(\mathcal{K}_{\lambda,f,M,B})$ is finite, there exist $N\in\mathbb{N}$ and $J_1,\dots,J_N\in\mathscr{C}_{k_{\lambda,f,M,\alpha}}(\mathcal{K}_{\lambda,f,M,B})$ such that $\mathscr{C}_{k_{\lambda,f,M,\alpha}}(\mathcal{K}_{\lambda,f,M,B})=\{J_1,\dots,J_N\}$. Note that
\begin{align*}
    \mathcal{K}_{\lambda,f,M,B}=\bigcup_{\nu=1}^N J_\nu.
\end{align*}
Pick any $\nu\in\{1,\ldots,N\}$. One obtains from Proposition~\ref{prop:local-holder-final} that the map
\begin{align*}
    \pi_{\lambda,f,M,B}|_{J_\nu}:
    J_\nu\to \mathcal{I}\!\left(V_\nu\right)\cap \mathcal{F}_B,
\end{align*}
for some finite word $V_\nu\in\bigcup_{n\in\mathbb{N}}\{1,\ldots,B\}^n$, is surjective and $\alpha$-H\"older. By Lemmas~\ref{lem:bounded-prefix-dimension} and~\ref{lem:holder-dimension-estimate}, one obtains
\begin{align*}
    \dim{\mathcal{F}_B}
    =\dim{\left(\mathcal{I}\!\left(V_\nu\right)\cap \mathcal{F}_B\right)}
    \leq\frac{1}{\alpha}\dim{J_\nu}.
\end{align*}
By the finite stability of Hausdorff dimension, 
\begin{align*}
    \dim{\mathcal{K}_{\lambda,f,M,B}}
    = \max_{\nu\in\{1,\ldots,N\}}\dim{J_\nu}
    \geq \alpha\dim{\mathcal{F}_B}.
\end{align*}
\end{proof}

\begin{prop}
\label{prop:eventual-full-dimension}
Let $\lambda\in\mathbb N$ and $f:\mathbb{N}\to\mathbb{N}$ be a function. Suppose $f$ is admissible. Then $\dim{\mathfrak{F}_{\lambda,f}}=1$.
\end{prop}
\begin{proof}
By Propositions~\ref{prop:run-control-final} and~\ref{prop:dimension-lower-final}, for any $M,B\in\mathbb{N}$, if $B\geq \lambda+2$ then
\begin{align*}
    \dim{\mathfrak{F}_{\lambda,f}}
    \geq \dim{\mathcal{K}_{\lambda,f,M,B}}
    \geq\sup_{0<\alpha<M/(M+1)}\alpha\dim{\mathcal{F}_B}
    =\frac{M}{M+1}\dim{\mathcal{F}_B}.
\end{align*}
By the classical fact~\cite{hensley1992continued,jarnik1928metrischen} that $\lim_{B\to\infty}\dim{\mathcal{F}_B}=1$, one obtains that for any $M\in\mathbb{N}$,
\begin{align*}
    \dim{\mathfrak{F}_{\lambda,f}}\geq\frac{M}{M+1}.
\end{align*}
Therefore, $\dim{\mathfrak{F}_{\lambda,f}}\geq1$, and the reverse inequality is trivial.
\end{proof}

\begin{proof}[Proof of Theorem~\ref{thm:main-sn-final}]
Note that $\mathfrak{F}_{\lambda,f}=\bigcup_{N\in\mathbb{N}}\mathfrak{F}_{\lambda,f,N}$, where
\begin{align*}
    \mathfrak{F}_{\lambda,f,N}
    &\coloneqq
    \left\{x\in[0,1)\setminus\mathbb{Q}:
    \lim_{n\to\infty}\frac{L_n(x,\lambda)}{f(n)}=1
    \text{ and }\right.\\
    &\qquad\qquad\left.\text{for any $m\in\mathbb{N}$ and $\mu\in\mathbb{N}\setminus\{\lambda\}$, if $m\geq N$ then }\frac{L_m(x,\lambda)}{L_m(x,\mu)}>1\right\}.
\end{align*}
By Proposition~\ref{prop:eventual-full-dimension} and the countable stability of Hausdorff dimension, $\sup_{N\in\mathbb{N}}{\dim{\mathfrak{F}_{\lambda,f,N}}}=1$.

Pick any $N\in\mathbb{N}$, $y\in\mathfrak{F}_{\lambda,f,N}$. Define $x\coloneqq T_{\lambda^N}(y)$. One claims that for any $\mu\in\mathbb{N}\setminus\{\lambda\}$ and $n\in\mathbb{N}$, $L_n(x,\lambda)>L_n(x,\mu)$. In the case of $n<N$, the first $n$ partial quotients of $x$ are all equal to $\lambda$. One obtains $L_n(x,\lambda)=n>L_n(x,\mu)=0$. In the case of $n\geq N$ and $m\coloneqq n-N<N$, the initial block $\lambda^N$ has length larger than the entire tail segment of length $m$. One obtains $L_n(x,\lambda)\geq N>m\geq L_n(x,\mu)$.

In the case of $n\geq N$ and $m\geq N$, one obtains that every $\mu$-run among the first $N+m$ partial quotients of $x$ is contained entirely in the tail corresponding to the first $m$ partial quotients of $y$. Thus, $L_n(x,\mu)=L_m(y,\mu)$. Since $y\in\mathfrak{F}_{\lambda,f,N}$ and $m\geq N$, one obtains $L_m(y,\lambda)>L_m(y,\mu)$. Every $\lambda$-run among the first $m$ partial quotients of $y$ also occurs among the first $N+m$ partial quotients of $x$; hence, $L_m(y,\lambda)\leq L_{N+m}(x,\lambda)=L_n(x,\lambda)$. Hence, $L_n(x,\lambda)\geq L_m(y,\lambda)>L_m(y,\mu)=L_n(x,\mu)$. The claim is proved for all cases.

Note that for any $m,N\in\mathbb{N}$,
\begin{align*}
    L_m(y,\lambda)
    \leq L_{N+m}(x,\lambda)
    \leq N+L_m(y,\lambda).
\end{align*}
Since $f$ is admissible, for any $N\in\mathbb{N}$ and sufficiently large $m\in\mathbb{N}$, $N\leq f(m)$ and 
\begin{align*}
    1\leq\frac{f(m+N)}{f(m)}\leq\frac{f(m+f(m))}{f(m)}.
\end{align*}
Thus, $\lim_{m\to\infty}f(m+N)/f(m)=1$. By $\lim_{m\to\infty}{L_m(y,\lambda)}/{f(m)}=1$, one obtains from the squeeze theorem that $\lim_{m\to\infty}{L_{N+m}(x,\lambda)}/{f(N+m)}=1$; hence $x\in{\mathfrak E}_{\lambda,f}$, and for any $N\in\mathbb{N}$,
\begin{align*}
    T_{\lambda^N}\!\left(\mathfrak{F}_{\lambda,f,N}\right)
    \subset{\mathfrak E}_{\lambda,f}.
\end{align*}
By Lemma~\ref{lem:prefix-map-derivative}, one obtains that for any $N\in\mathbb{N}$, 
\begin{align*}
    \dim{{\mathfrak E}_{\lambda,f}}
    \geq \dim{T_{\lambda^N}\!\left(\mathfrak{F}_{\lambda,f,N}\right)}
    = \dim{\mathfrak{F}_{\lambda,f,N}}.
\end{align*}
Therefore, $\dim{\mathfrak{E}_{\lambda,f}}\geq1$, and the reverse inequality is trivial.
\end{proof}

\bibliographystyle{siam} 
\bibliography{name}

\end{document}